\documentclass[11pt]{amsart}
\usepackage{latexsym,enumerate}
\usepackage{amssymb, xcolor,mathrsfs}
\usepackage{amsmath,amsthm,amsfonts,amssymb,latexsym, mathabx}
\usepackage[colorlinks,pagebackref]{hyperref}
\usepackage{arydshln}

\newtheorem{theorem}{Theorem}[section]
\newtheorem{lemma}[theorem]{Lemma}

\newtheorem{corollary}[theorem]{Corollary}

\theoremstyle{definition}

\newtheorem{thml}{Theorem}

\theoremstyle{remark}

\numberwithin{equation}{section}

\newcommand{\GL}{{\mathrm {GL}}}
\newcommand{\PGL}{{\mathrm {PGL}}}
\newcommand{\SL}{{\mathrm {SL}}}
\newcommand{\PSL}{{\mathrm {PSL}}}
\newcommand{\POmega}{{\mathrm {P\Omega}}}

\newcommand{\PSU}{{\mathrm {PSU}}}

\newcommand{\SO}{{\mathrm {SO}}}

\newcommand{\Sp}{{\mathrm {Sp}}}
\newcommand{\PSp}{{\mathrm {PSp}}}

\newcommand{\Aut}{{\mathrm {Aut}}}

\newcommand{\Irr}{{\mathrm {Irr}}}

\newcommand{\diag}{{\mathrm {diag}}}

\newcommand{\Syl}{{\mathrm {Syl}}}
\newcommand{\St}{{\mathsf {St}}}

\newcommand{\Sym}{{\mathrm {Sym}}}
\newcommand{\Alt}{{\mathrm {Alt}}}

\newcommand{\ZZ}{{\mathbb Z}}

\newcommand{\FF}{{\mathbb F}}

\newcommand{\wt}{\widetilde}

\newcommand{\bC}{{\mathbf{C}}}

\newcommand{\bO}{{\mathbf{O}}}
\newcommand{\bN}{{\mathbf{N}}}

\newcommand{\tw}[1]{{}^#1}

\def\nor{\trianglelefteq\,}
\def\irr#1{{\rm Irr}(#1)}
\def\sbs{\subseteq}
\def\syl#1#2{{\rm Syl}_#1(#2)}
\def\oh#1#2{{\bf O}_{#1}(#2)}

\def\zent#1{{\bf Z}(#1)}
\def\norm#1#2{{\bf N}_{#1}(#2)}
\def\cent#1#2{{\bf C}_{#1}(#2)}

\def\irrp#1#2{{\rm Irr}_{#1'}(#2)}
\newcommand{\aut}{\operatorname{Aut}}
\newcommand{\type}{\operatorname}

\begin{document}

\title[Principal blocks with six characters]
{Principal blocks with six\\ ordinary irreducible characters}

\author[N.\,N. Hung]{Nguyen N. Hung}
\address[N.\,N. Hung]{Department of Mathematics, The University of Akron, Akron,
OH 44325, USA} \email{hungnguyen@uakron.edu}

\author[A.\,A. Schaeffer Fry]{A.\,A. Schaeffer Fry}
\address[A.\,A. Schaeffer Fry]{Dept. Mathematics and Statistics, MSU Denver, Denver, CO 80217, USA}
\email{aschaef6@msudenver.edu}

\author{Carolina Vallejo}
\address[C. Vallejo]{Dipartimento di Matematica e Informatica `Ulisse Dini', 50134 Firenze (Italy)}
\email{carolina.vallejorodriguez@unifi.it}

\thanks{The first author is grateful for the support of an UA
Faculty Research Grant. The second-named author is grateful for the
support of a grant from the National Science Foundation, Award No.
DMS-2100912. The third author, as part of the GNSAGA, is grateful for the support of the {\em Istituto Nazionale di Alta Matematica} (INDAM).
The authors also thank Gunter Malle for comments on an earlier draft.}


%\date{}

\begin{abstract}
We classify Sylow $p$-subgroups of finite groups whose principal
$p$-blocks have precisely six ordinary irreducible characters.
\end{abstract}

\maketitle

%%%%%%%%%%%%%%%%%%%%%%%%%%%%%%%%%%%%%%%%%%%%%%%%%%%%%%%%%%%%%%%%%%%%%%

\centerline{\sf In memory of Georgia Benkart, an inspiration to all mathematicians.}

\section{Introduction}

Classifying finite groups with a given (relatively small) number of
conjugacy classes is a classical and natural problem in finite group
theory \cite{VV85, SV07}. Such classification not only provides a
handy source for testing several predictions/conjectures involving
the class number, but also plays a critical role in the proofs of
theorems that require ad hoc arguments for groups with a small
number of conjugacy classes, see \cite{HHKM11,Maroti16} for
instance.

Given a positive integer $k$, it is expected that there are finitely
many isomorphism classes of groups that can
occur as defect groups of blocks with $k$ ordinary irreducible
characters (this, often referred to as Brauer's Problem 21
\cite{Brauer}, has been shown by K\"{u}lshammer and Robinson
\cite{Kulshammer-Robinson} to be a consequence of the Alperin-McKay
conjecture \cite[Conjecture~9.5]{Navarro18}). The problem of
determining all possible structures of a defect group of a block
that has a given number of irreducible characters can be viewed as
the modular analogue of the above problem on conjugacy classes. As
one should anticipate, the modular problem is much harder and
largely open -- for instance, it is still unknown whether a block
has precisely $3$ ordinary characters if and only if its defect
group has order $3$ (see \cite[Section 4]{KNST14} and
Table~\ref{table:1} at the end of this paper).

Let $p$ be a prime. Recall that the principal $p$-block of a finite
group $G$ is the one containing the principal character ${\bf 1}_G$
and that its defect groups are the Sylow $p$-subgroups of $G$. Thanks to
the recent work of Koshitani-Sakurai \cite{KS21} and of
Rizo-Schaeffer Fry-Vallejo \cite{RSV21}, those Sylow $p$-subgroups
of finite groups with principal $p$-blocks having up to five
irreducible characters have been completely determined. This in turn
has contributed to the solution of the principal-block case
\cite{Hung-SchaefferFry} of H\'{e}thelyi-K\"{u}lshammer's conjecture
\cite{Hethelyi-kulshammer} and to obtaining a $p$-local lower bound
for the number of height-zero characters in principal blocks
\cite{HSV21}.

The purpose of this paper is to advance on the determination of the
structures of defect groups of blocks with a given number of ordinary
characters. Our main result classifies the defect groups of
principal blocks with six characters.

\begin{thml}\label{thm:kB0=6}
Let $G$ a finite group, $p$ a prime, and $P\in\Syl_p(G)$. Suppose
that the principal $p$-block of $G$ has precisely six irreducible ordinary
characters. Then $|P|=9$.
\end{thml}

Let $B_0(G)$, or sometimes just $B_0$, denote the principal
$p$-block of $G$, and let  $k(B_0)$ denote the number of ordinary
irreducible characters of $B_0$. Our strategy for proving Theorem
\ref{thm:kB0=6}, which is somewhat different from the above-mentioned previous
work on $k(B_0)\in \{4,5\}$, is to analyse the number $k_0(B_0)$ of
those characters in $B_0$ of \emph{height zero}. One of the main
results of \cite{HSV21} already classifies the Sylow structure of
finite groups with $k_0(B_0)\leq 5$, and perhaps surprisingly, none
of these possibilities could occur when $k(B_0)=6$. Therefore, we
are left to deal with $k(B_0)=k_0(B_0)=6$, in which case the Sylow
subgroup $P$ must be abelian, by the recent proof of Brauer's height
zero conjecture for principal blocks by Malle and Navarro
\cite{MN21} (note that the more general case of the conjecture has
now also been proved \cite{MNST}).

Our proof makes use of the Classification of Finite Simple Groups. In
particular, we have to find lower bounds for the number of
$\aut(S)$-orbits of irreducible characters in the principal block of
a non-abelian simple group $S$. This has been recently studied in
connection with other problems on character degrees and character
bounds \cite{Martinez,Hung-SchaefferFry}. With this, we are able to
restrict ourselves to studying the $\aut(S)$-orbits of characters in
the principal $3$-blocks of such an $S$ with an abelian Sylow
$3$-subgroup, see Theorem \ref{thm:simplegroups3}.

%As was the case in the main results of \cite{KS21} and \cite{RSV21},
%it can be shown that the statement of our main theorem above follows
%from the principal block case of the Alperin-McKay conjecture.

There are two isomorphism classes of groups of order $9$, namely
${\sf {C}}_9$ and ${\sf C}_3 \times {\sf C}_3$. Either of them can
occur as the defect group of a principal block with 6 ordinary
characters, as shown in semidirect products of ${\sf C}_2$ acting on
${\sf {C}}_9$ and ${\sf C}_3 \times {\sf C}_3$ by inversion. With that being said, the `full inverse' of Theorem \ref{thm:kB0=6} is not
always true: if $|P|=9$ then $k(B_0)$ could be either $6$ or $9$
(see the proof of Theorem~\ref{thm:kB0=6-main-repeated}). The next
result, which fully characterizes finite groups with $k(B_0)=6$ in
terms of $p$-local structure, offers a more complete version of
Theorem~\ref{thm:kB0=6}.

\begin{thml}\label{thm:kB0=6-main}
Let $G$ a finite group, $p$ a prime, and $P\in\Syl_p(G)$. Let $B_0$
denote the principal $p$-block of $G$. Then $k(B_0)=6$ if and only
if precisely one of the following holds:
\begin{itemize}
\item[(i)] $P= {\sf C}_9$ and $|\norm G P :\cent G P |=2$.

\item[(ii)] $P={\sf C}_3 \times {\sf C}_3$ and either $\norm G P /\cent G P \in\{{\sf C}_4, {\sf
Q}_8\}$ or $\norm G P /\cent G P\cong {\sf C}_2$ acts fixed-point
freely on $P$.

\end{itemize}
\end{thml}

Proofs of the main theorems are contained in Section \ref{sec:3} and
the necessary results on finite simple groups are proved in Section
\ref{sec:2}. At the end of the paper, we provide a summary of what is known about the structure of defect groups of small blocks in Table \ref{table:1}.

%\hungcomment{I think the following is nice and close to what we had,
%but requires a lot more work. I just realized that it is the
%``abelian case" of Conjecture B of Navarro-Sambale-Tiep. Not sure if
%we should mention this in the current paper. The only if implication
%is fine. It remains to show that if $k(B_0)=k_0(B_0)=9$ then
%$|P|=9$. The proof of this would be a repetition of what we have
%done, unless there is a new idea.}
%\carolinacomment{Conjecture B of Navarro-Sambale-Tiep is stated just for $p=3$.}
%
%
%\hungcomment{Let $G$ a finite group, $p$ a prime, $P\in\Syl_p(G)$,
%and $B_0$ denote the principal $p$-block of $G$. Then $|P|=9$ if and
%only if $k(B_0)=k_0(B_0)\in\{6,9\}$. }
%\carolinacomment{If $G={\sf C}_{19}\rtimes {\sf C}_3$ is not nilpotent, then for $p=19$ we have that $k_0(B_0)=k(B_0)=9.$}

%%%%%%%%%%%%%%%%%%%%%%%%%%%%%%%%%%%%%%%%%%%%%%%%%%%%%%%%%%%%%%%%%%%%%%%

\section{Principal blocks of simple groups}\label{sec:2}

In this section, we prove the following statements on simple groups,
which will be needed for the proof of our main theorems.

\begin{theorem}\label{thm:simplegroups3}
Let $S$ be a non-abelian simple group. Let $p=3$ and let
$B_0:=B_0(S)$ be the principal 3-block of $S$. Assume that $Q \in
\syl p S$ is abelian and $|Q|\geq 9$. If $S \leq A \leq \aut (S)$,
then one of the following holds:
\begin{enumerate}
\item[{\rm (a)}] The action of $A$ defines at least $4$ orbits on $\irr{B_0}\setminus \{ {\bf 1}_S \}$.
\item[{\rm (b)}] $S \in \{ {\rm PSL}_2(q), {\rm PSL}_3(q), {\rm PSU}_3(q) \}$ with  $(3, q)=1$, %$Q\cong {\sf C}_9$,
$(|A:S|, 3)=3$, and the action of $A$ induces $3$ orbits on
$\irr{B_0(S)}\setminus \{ {\bf 1}_S\}$. In this case, $k(B_0(A))>6$.
\item[{\rm (c)}] $S=\Alt_6=\PSL_2(3^2)$, $A$ has a subgroup isomorphic to ${\rm
M}_{10}$, and the action of $A$ induces $3$ orbits on
$\irr{B_0(S)}\setminus \{ {\bf 1}_S\}$.
\end{enumerate}
\end{theorem}

%\textcolor{red}{For $\Alt_6$ there are only 3 nontrivial $\aut(S)$-orbits on $\Irr(B_0(S))$,
%and one of the choices of $A$, namely $A6.2_3$ in GAP, has only 6 members in $B_0(A)$}////

\begin{theorem}\label{thm:simplegroups}
Let $S$ be a non-abelian simple group. Let $p$ be an odd prime and
let $B_0:=B_0(S)$ be the principal $p$-block of $S$. Assume that $P
\in \syl p S$ is abelian. If $k(B_0)=6$, then $p=3$ and $|P|=9$.
\end{theorem}

We refer the reader to \cite[Chapter~9]{nbook} for basics on the
block theory involving normal subgroups and quotient groups. Recall
that if $N$ is a normal subgroup of $G$ and $B$ and $b$ are blocks
of $G$ and $N$ respectively, then $B$ is said to \emph{cover} $b$ if there are
$\chi\in{\rm Irr}(B)$ and $\theta\in{\rm Irr}(b)$ such that $\theta$
is an irreducible constituent of the restriction $\chi_N$. It is
clear that $B_0(G)$ covers $B_0(N)$. For $\theta\in\Irr(N)$, we
write $\Irr(G | \theta)$, respectively $\Irr(B| \theta)$, for the
set of those characters of $G$, respectively $B$, containing
$\theta$ as a constituent when restricted to $N$.

\begin{lemma}\label{lem:blockabove}
Let $G$ be a finite group, $N \nor G$, and $p$ a prime. Let $B_0$ be
the principal $p$-block of $G$.
\begin{itemize}

\item[(i)]  $\irr {B_0(G/N)}\subseteq \irr{B_0}$.

\item[(ii)] If $N$ is a $p'$-group, then $\irr{B_0(G/N)}=\irr{B_0}$.

\item[(iii)] For every $\theta\in\irr{B_0(N)}$, there exists
$\chi\in\irr{B_0|\theta}$.

\item[(iv)] Suppose that $B\in {\rm Bl}(G)$ is the only block
covering $b \in {\rm Bl}(N)$. Then for every $\theta \in \irr b$, we
have $\irr{G| \theta}\sbs \irr{B}$.

\item[(v)] $\irr{B_0}\cap \irr {G/N}$ is a union of blocks of $G/N$.
\end{itemize}
\end{lemma}

\begin{proof}
Part (i) and (v) follow as $B_0$ dominates $B_0(G/N)$, see the
discussion on \cite[pp. 198-199]{nbook} on block dominance and
inclusion. Part (ii) is \cite[Theorem 9.9.(c)]{nbook}. Part (iii) is
\cite[Theorem 9.4]{nbook}. Part (iv) can be found in \cite[Lemma
1.2]{RSV21}, for instance.
\end{proof}

We begin by proving Theorem \ref{thm:simplegroups3} in various cases.

\begin{lemma}\label{lem:sporalt}
Theorem \ref{thm:simplegroups3} holds if $S$  is a sporadic group,
the Tits group $\tw{2}\type{F}_4(2)'$,  a group of Lie type with
exceptional Schur multiplier, or an alternating group.
\end{lemma}

\begin{proof}
First, if $S$ is a sporadic group, the Tits group
$\tw{2}\type{F}_4(2)'$,  a group of Lie type with exceptional Schur
multiplier (see \cite[Table 6.1.3]{GLS} for a list), or $\Alt_6$, then the statement can be verified using
GAP \cite{GAP} and the information available in its Character Table
Library.  In particular, we see from this that case (c) occurs.

So, we now suppose that $S=\Alt_n$ is an alternating group with  $n>
6$. Then $A\in \{\Alt_n, \Sym_n\}$, and it suffices to show that
$\Irr(B_0(\Sym_n))$ contains at least 5 characters whose restrictions
to $\Alt_n$ are distinct. Recall that the characters of $\Sym_n$ are
in bijection with partitions of $n$, and non-conjugate partitions
yield distinct restrictions to $\Alt_n$. The characters in
$B_0(\Sym_n)$ are those with $3$-core $(r)$, where $n=3m+r$, $0\leq
r< 3$ (see, e.g. \cite[Theorem (11.1)]{Olsson93}).  Now, we see that there are at least 5 non-conjugate
partitions with $3$-core $(r)$ for $m\geq 2$, yielding the claim.
%
%The characters in $B_0(\Sym_n)$ are further in bijection with those in $B_0(\Sym_{3m})$
%where $n=3m+r$, $0\leq r\leq 2$ (see \cite[Theorem (1.10)]{MO83}) ////
%\textcolor{red}{does this preserve non-conjugate? if not may need argue $k(s,t)\geq 10$ if $m\geq 3$,
%and note explicitly that for $m=2$ we have 5 nonconjugate partitions}///, so we may assume $3\mid n$.
%Now, the characters in $B_0(\Sym_{3m})$ are those corresponding to partitions with trivial $3$-core,
%and we see there are at least 5 non-conjugate partitions with trivial $3$-core for $m\geq 2$, yielding the claim.
\end{proof}

We say that $\theta \in \irr N$, where $N \nor G$, \emph{extends} to $G$ (or is \emph{extendable} to $G$) if
there is some $\chi \in \irr G$ such that $\chi_N=\theta$. In that
case,
 $$\irr{G|\theta}=\{ \beta \chi \ | \ \beta \in \irr{G/N}\}$$
by a theorem of Gallagher \cite[Corollary 6.17]{Isaacs}.

The next observation will be useful in  some of the remaining cases.

\begin{lemma}\label{lem:2charabove}
Let $X$, $Y$, $\wt{X}$, and $E$ be finite groups such that $X\lhd
\wt{X}\lhd \wt{X}E$; $X\lhd Y\leq \wt{X}E$; and $E$ is abelian.
Suppose further that $p$ is a prime such that $[\wt{X}:X]=p$ and
$p\mid [Y:X]$.  Let $\wt\chi\in\irr{B_0(\wt{X})}$ be a character in
the principal $p$-block of $\wt{X}$ that is extendable to $\wt{X}E$
and restricts to $\chi\in\irr{X}$.  Then there exist at least two
characters in $\irr{B_0(Y)}$ lying above $\chi$.
\end{lemma}

\begin{proof}
Note that our assumptions imply that $Y/(\wt{X}\cap Y)$ is abelian
and $\wt{X}\cap Y\in \{\wt{X}, X\}$. Let $X\lhd Y_p\lhd Y$ such that
$Y_p=\wt{X}$ if $\wt{X}\cap Y=\wt{X}$ and such that $|Y_p/X|=p$ if
$\wt{X}\cap Y=X$. Then in either case, we have $|Y_p/X|=p$ and
$Y/Y_p$ is abelian.

Note that $\chi$ has $p$ distinct extensions $\chi_1,\ldots,\chi_p$
to $Y_p$, which all must lie in $B_0(Y_p)$ by Lemma
\ref{lem:blockabove}(iv) and the fact that $[Y_p:X]=p$, so
$B_0(Y_p)$ is the unique $p$-block of $Y_p$ above $B_0(X)$. Since
$\chi$ extends to $\wt{X}E$, and hence $Y$, at least one of these,
say $\chi_1$, extends to $Y$. In the case that $Y_p=\wt{X}$, we may
even specify $\chi_1:=\wt{\chi}$. Since $Y/Y_p$ is abelian, it
follows that every character in $\irr{Y|\chi_1}$ is an extension, by
Gallagher's theorem \cite[Corollary 6.17]{Isaacs}. In particular,
$\chi_1$ extends to some member of $B_0(Y)$ lying above $\chi$, by
Lemma \ref{lem:blockabove}(iii). But, note that by Lemma
\ref{lem:blockabove}(iii), there is also a member of
$\irr{B_0(Y)|\chi_2}$, also lying above $\chi$, but that this
character cannot lie above $\chi_1$ since $\chi_1$ and $\chi_2$ are
not $Y$-conjugate.  This yields at least two distinct members of
$\irr{B_0(Y)|\chi}$, as desired.
\end{proof}

In what follows, for $\epsilon\in\{\pm1\}$, we will use
$\PSL_n^\epsilon(q)$ to denote the group $\PSL_n(q)$ of type
$\type{A}_{n-1}$ if $\epsilon=1$ and the group $\PSU_n(q)$ of type
$\tw{2}\type{A}_{n-1}$ if $\epsilon=-1$, and we will use analogous
notation for the related groups $\SL_n^\epsilon(q)$,
$\GL_n^\epsilon(q)$, and $\PGL_n^\epsilon(q)$.

\begin{lemma}\label{lem:PSL23nondef}
Theorem \ref{thm:simplegroups3} holds when $S$ is one of the groups
$\PSL_2(q)$ or $\PSL_3^\epsilon(q)$ with $q=q_0^f$ a power of a
prime $q_0\neq 3$.
\end{lemma}

\begin{proof}
Let $S=\PSL_n^\epsilon(q)$ with $n\in\{2,3\}$ and $q=q_0^f$ with
$q_0\neq 3$ a prime. Write \[G:=\SL_n^\epsilon(q),
\wt{S}:=\PGL_n^\epsilon(q), \text{ and } \wt{G}:=\GL_n^\epsilon(q)\]
for the appropriate choice of $n, \epsilon$. In this case, the dual group $\wt{G}^\ast$ is isomorphic to $\wt{G}$, and we identify the two groups.
Note that
$\aut(S)=\wt{S}\rtimes D$, where $D$ is an appropriate group
generated by field and graph automorphisms. (See e.g. \cite[Theorem
2.5.12]{GLS}.) In this case, $D$ is further abelian. Let
$e\in\{1,2\}$ be the order of $\epsilon q$ modulo $3$.  The
unipotent characters of $\wt{G}$ (or $\wt{S}$, $S$) are in bijection
with partitions of $n$, and by \cite{FS}, two such characters lie in
the same block if and only if they correspond to partitions with the
same $e$-core. In the case $e=1$ and $S=\PSL_3^\epsilon(q)$, we see
from this that there are two nontrivial unipotent characters in
$B_0(S)$, and these are $\aut(S)$-invariant (see e.g. \cite[Theorem
2.5]{Malle08}). In the remaining cases, there is one nontrivial
unipotent character in $B_0(S)$ (namely, the Steinberg character
$\mathbf{St}_S$), which is again $\aut(S)$-invariant.

\medskip

(I) Suppose first that $|Q|>9$.

\medskip

(Ia) If $S=\PSL_3^\epsilon(q)$ with $e=1$, then this means that
$9\mid (q-\epsilon)$ and all three unipotent characters lie in the
principal block.  Let $a_1, a_2\in C_{q-\epsilon}\leq
\FF_{q^2}^\times$ such that $|a_1|=3$ and $|a_2|=9$.  Then for
$i=1,2$, let \[s_i:=\diag(a_i, a_i^{-1}, 1)\in \wt{G}.\]   Then each
$s_i$ defines a semisimple character $\chi_{s_i}$ of $\wt{G}$ that
lies in $B_0(\wt{G})$ using \cite[Theorem 9.12]{CE04}, is trivial on
$\zent{\wt{G}}$ since $s_i\in [\wt{G}, \wt{G}]=G$ (see e.g.
\cite[Prop. 2.7]{SFT21}), and such that $s_1^\alpha z$ is not
$\wt{G}$-conjugate to $s_2$ for any $z\in\zent{\wt{G}}$ and
$\alpha\in D$ (since semisimple classes in $\wt{G}$ are determined
by their eigenvalues and the eigenvalues of $s_1^\alpha$ still have
order 3). Further, there is an isomorphism $z\mapsto \hat z$ between
$\zent{\wt{G}}$ and $\irr{\wt{G}/G}$, such that
$\chi_{sz}=\chi_{s}\hat{z}$ in this situation for $s\in \wt{G}$
semisimple. (See \cite[(8.19) and Proposition 8.26]{CE04}).  Then
since $\chi_{s_1}^\alpha=\chi_{s_1^\alpha}$ by \cite[Corollary
2.5]{NTT08},
%(here we abuse notation and identify $\alpha^\ast$ with
%$\alpha$ in the notation of loc. cit. since $\wt{G}^\ast\cong\wt{G}$),
we see $\chi_{s_1}^\alpha$ and $\chi_{s_2}$ must
necessarily have distinct restrictions to $S=G/\zent{G}\cong
G\zent{\wt{G}}/\zent{\wt{G}}$, and we have obtained at least two
additional $\aut(S)$-orbits in $\Irr(S)$ by restriction.

\medskip

(Ib) Now let $S=\PSL_3^\epsilon(q)$ with $e=2$ or $S=\PSL_2(q)$, so
that $Q$ is cyclic and the condition $|Q|>9$ means $27\mid (q^2-1)$.
Let $a_1, a_2, a_3\in \FF_{q^2}^\times$ with orders $|a_i|=3^i$ for
$i=1,2,3$. Then considering $s_i\in\wt{G}$ whose nontrivial
eigenvalues are $\{a_i, a_i^{-1}\}$, we again obtain
$\chi_{s_i}\in\irr{B_0(\wt{G})}$, trivial on $\zent{\wt{G}}$, and
hence $\chi_{s_i}$ may be identified with a character in
$\irr{B_0(\wt{S})}$.  Suppose now for a contradiction that $\chi_{s_i}^\alpha \hat{z}$ restricts to the same character of $S$ as $\chi_{s_j}$ for $i\neq j\in\{1,2,3\}$, some $\alpha\in D$, and
$ z\in\zent{\wt{G}}$.  This means that $s_i^\alpha z$ is
$\wt{G}$-conjugate to $s_j$ and that the character $\hat{z}$ must be trivial on $\zent{\wt{G}}$. Then we see that the corresponding $z$ (and hence $\hat{z}$) must have order
a nontrivial power of $3$, contradicting that $3$ does not divide
$|\wt{S}/S|$ in the cases being considered.  This yields our
 additional 3 $\aut(S)$-orbits in this
case.

\medskip

(II) Finally, assume that we are in the last situation:
$S=\PSL_2(q)$ or $\PSL_3^\epsilon(q)$ and  $|Q|=9$, and further assume that $A$ defines
fewer than 4 orbits on $\irr{B_0(S)}\setminus {\bf 1}_S$.

\medskip

(IIa) First, if $e=1$ and $S=\PSL_3^\epsilon(q)$, this means that
$3\mid\mid(q-\epsilon)$.
In this case, we still have two nontrivial, $\aut(S)$-invariant,
unipotent characters in $\irr{B_0}$.  Taking $s_1=\mathrm{diag}(a_1,
a_1^{-1}, 1)$ as before with $|a_1|=3$, the corresponding semisimple
character $\chi_{s_1}$ of $\wt{G}$ is trivial on $\zent{\wt{G}}$,
lies in $B_0(\wt{S})$, and restricts to the sum of three characters
in $B_0(S)$. So, if $S\lhd A\leq \aut(S)$ with $3\nmid [A:S]$, these
characters must also be invariant under $A$, giving more than 4
$A$-orbits on $\Irr(B_0(S))$.  So, assume that $3\mid [A:S]$.  By
\cite[Theorems 2.4 and 2.5]{Malle08}, the unipotent characters
extend to $\aut(S)$. Now, by applying Lemma \ref{lem:2charabove}
with $(X, \wt{X}, Y, E)=(S, \wt{S}, A, D)$ to each unipotent
character, we obtain
%\textcolor{blue}{expand and make lemma for this argument!!:  and the three extensions to $B_0(\wt{S})$ (indeed, recall $|\wt{S}/S|=3$ so there is a unique block of $\wt{S}$ above $B_0(S)$) must have characters lying above them in $B_0(\aut(S))$, and at least one of them extends to this character.}  This yields that $B_0(A)$ must have
at least 6 characters just from those above the three unipotent
characters, and hence more than 6 in total. Hence we are in the situation of (b).

\medskip

(IIb) We are left with the case that  $9\mid\mid(q^2-1)$, so that
$9\mid\mid(q-\eta)$ for some $\eta\in\{\pm1\}$ and $Q$ is cyclic of
size $9$. (In particular, we have $e=2$ and $\eta=-\epsilon$ in case
$S=\PSL_3^\epsilon(q)$.) Here the only nontrivial unipotent
character in $B_0({S})$ is the Steinberg character $\mathbf{St}$, and there is a
unique unipotent block of $\wt{G}$ with positive defect.   Let $a_1,
a_2\in\FF_{q^2}^\times$ and $s_1, s_2\in\wt{G}$ be defined
exactly as in the case (Ib) above. Then $\chi_{s_i}\in
\Irr(B_0(\wt{S}))$ and $\chi_{s_1}$ and $\chi_{s_2}$ lie in distinct
$\aut(S)$-orbits, as before. Note that each $\chi_{s_i}$ is
irreducible on $S$ using the same arguments as before, and that
$\chi_{s_1}$ is further $D$-invariant, since any element of $D$
either inverts or stabilizes the eigenvalues of order $3$. From the
restrictions of $\chi_{s_1}$ and $\chi_{s_2}$ to $S$, in addition to $\mathbf{St}$, this yields three $\aut(S)$-orbits on $\irr{B_0(S)}\setminus \{
{\bf 1}_S \}$. Then since we have assumed we do not have four
$A$-orbits on this set, the remaining characters in $\irr{B_0(S)}$
must be $A$-conjugate to the restriction of $\chi_{s_2}$ to $S$.
But note this means the three choices of pairs $\{a_2, a_2^{-1}\}$
with $|a_2|=9$ must yield $A$-conjugate characters, say $\chi_{s_2},
\chi_{s_2'}, \chi_{s_2''}$, and hence $3$ divides $ |A/S|$.

Hence, we see that if $A$ is an almost simple group with socle $S$
permitting only 3 orbits on $\irr{B_0}\setminus \{ {\bf 1}_S \}$,
then $3\mid |A/S|$ and $A/(\wt{S}\cap A)$ is abelian. Another
application of Lemma \ref{lem:2charabove} applied to $\mathbf{1}_S$,
$\mathbf{St}$, and $\chi_{s_1}$ now forces at least 6 characters in
$\irr{B_0(A)}$, along with at least one more above $\chi_{s_2}$.
%\textcolor{blue}{also will be in new lemma: Let $\wt{S}\cap A\leq A_1\leq A$ with $[A/A_1]=3$. Note that $\mathbf{1}_{\wt{S}\cap A}, \mathbf{St}_{\wt{S}\cap A},$ and $(\chi_{s_1})_{\wt{S}\cap A}$ all extend to $A$, and hence there is some extension of each in $B_0(A_1)$.  Then since there is a unique block of $A$ above $B_0(A_1)$, we see $\irr{B_0(A)}$ necessarily contains at least 6 characters from further extensions to $A$. }
\end{proof}

\begin{corollary}\label{cor:PSL23}
Theorem \ref{thm:simplegroups3} holds when $S$ is one of the groups
$\PSL_2(q)$ or $\PSL_3^\epsilon(q)$ with $q=q_0^f$ a power of a
prime $q_0$.
\end{corollary}

\begin{proof}
From Lemma \ref{lem:PSL23nondef}, we may assume that $q_0=3$.  Then
the condition that $Q$ is abelian and $|Q|\geq 9$ means that
$S=\PSL_2(3^f)$ with $f\geq 2$. (See, e.g.
\cite[Theorem]{shenzhau}.) Since the case $f=2$ is  covered by Lemma
\ref{lem:sporalt}, we assume that $f>2$.

Here $B_0(S)$ contains all irreducible characters, aside from the
Steinberg character (see \cite[Theorems 1.18 and 3.3]{Cabanes18}).
Then we see from the well-known character table for $S$ that there
are three distinct character degrees in $\irr{B_0(S)}\setminus\{{\bf
1}_S\}$, and it suffices to show that there are two semisimple
characters of the same degree that are not $\aut(S)$-conjugate.  We
will employ similar strategies to the second paragraph of the proof
of Lemma \ref{lem:PSL23nondef}.

Note that since $q=3^f\geq 27$, at least one of $q-\eta$ for
$\eta\in\{\pm1\}$  is a composite number of the form $4m$ with
$m\geq 7$. Then $C_{q-\eta}\leq\FF_{q^2}^\times$ contains two
elements $\zeta_1, \zeta_2$ of order larger than 4 and satisfying
$|\zeta_1|\not\in\{ |\zeta_2|, 2|\zeta_2|\}$. For $i=1,2$, let $s_i$ be a semisimple element of $\GL_2(q)$ with eigenvalues $\{\zeta_i,
\zeta_i^{-1}\}$. Then the semisimple
characters $\chi_{s_i}$ of $\GL_2(q)$ corresponding to the
 $s_i$   will restrict irreducibly to $\SL_2(q)$ (since
$s_iz$ cannot be conjugate to
$s_i$ for any $1\neq z\in \zent{\GL_2(q)}$) and be trivial on the
center (since $s_i\in \SL_2(q)=[\GL_2(q), \GL_2(q)]$). Further, the
restrictions of $\chi_{s_1}$ and $\chi_{s_2}$ to $\PSL_2(q)$ cannot
be conjugate under field automorphisms since
$\chi_{s_1}^\alpha=\chi_{s_1^{\alpha}}$ for $\alpha\in D$, where we write $\aut(S)=\wt{S}\rtimes D$ as in the proof of Lemma \ref{lem:PSL23nondef}, using
\cite[Corollary 2.5]{NTT08}, and ${s_1}^{\alpha}$ cannot be
conjugate to $s_2z$ for $z\in \zent{\GL_2(q)}$ of order dividing
$2$. (Recall that the $z\in\zent{\GL_2(q)}$ are in bijection with
$\hat{z}\in \Irr(\GL_2(q)/\SL_2(q))$, and that
$\chi_{s_2z}=\chi_{s_2}\hat{z}$. Further, if $|z|>2$, then $\hat{z}$
is
not trivial on $\zent{\GL_2(q)}$.)
This shows that the restrictions of $\chi_{s_1}$ and $\chi_{s_2}$ to $S$ are not $\aut(S)$-conjugate, as desired.
\end{proof}

We now complete the proof of Theorem \ref{thm:simplegroups3}, which
essentially follows from the observations in \cite[Section
3]{RSV21}.

\begin{proof}[Proof of Theorem \ref{thm:simplegroups3}]
From Lemma \ref{lem:sporalt} and Corollary \ref{cor:PSL23}, we may
assume that $S$ is a simple group of Lie type defined over $\FF_q$
with nonexceptional Schur multiplier, where $q$ is a power $q=q_0^f$
of a prime $q_0$ and that $S \not\in\{\PSL_2(q),
\PSL_3^\epsilon(q)\}$.  Further, $q_0\neq 3$, as otherwise $Q$ is
not abelian by \cite[Theorem]{shenzhau}, since $S\neq \PSL_2(q)$.

Assume first that $S$ is of exceptional type, including
$\tw{2}\type{F}_4(q)$ and $\tw{3}\type{D}_4(q)$.  Then the proof of
\cite[Lemma 3.7]{RSV21} yields more than 4 orbits under $\aut(S)$ in
$\irr{B_0(S)}$, and we are done in this case. (Note that we do not
need to consider the cases  $\tw{2}\type{G}_2(q)$ and
$\tw{2}\type{B}_2(q)$, since we would have $q_0=3$ in the first case and $3$
does not divide $|S|$ in the second.)

We are left with the case that $S$ is of classical type
$\type{A}_{n-1}$ or  $\tw{2}\type{A}_{n-1}$ with $n\geq 4$,
$\type{B}_n$ with $n\geq 3$, $\type{C}_n$ with $n\geq 2$, or
$\type{D}_n$ or $\tw{2}\type{D}_n$ with $n\geq 4$. Here, the proof
of \cite[Lemma 3.8]{RSV21} shows that we can assume $me\leq 4$ when
$S=\PSL_n^\epsilon(q)$, where $e\in\{1,2\}$ is the order of
$\epsilon q$ modulo 3 and $n=me+r$ with $0\leq r<e$.  Similarly, the
proof of \cite[Lemma 3.10]{RSV21} shows that there are at least 5
$\aut(S)$-orbits in $B_0(S)$ in the remaining classical cases,
except possibly if $S= \type{C}_2(q)$.  (Note that since $p=3$, we
have $e=1$ and $n=m$ for these remaining classical cases in the
notation of loc. cit.) If $S=\type{C}_2(q)$, then there is at worst
one pair of unipotent characters in $B_0(S)$ that are interchanged
by elements of $\aut(S)$ (see \cite[Theorem 2.5]{Malle08}). Then the
same arguments as in \cite[Lemma 3.10]{RSV21} again finish this
case, since $k(2,2)=5$ (in the notation of loc. cit).

Finally, we assume $S=\PSL_4^\epsilon(q)$ or $\PSL_5^\epsilon(q)$
with $em=4$. Recall that by \cite{FS}, unipotent characters lie in
the same block if and only if they correspond to partitions with the
same $e$-core. We see that there are 5 partitions of $4$ with
trivial $e$-core and similarly 5 partitions of $5$ with $e$-core
(1), and hence there exist 5 unipotent characters in $B_0(S)$, which
are $\aut(S)$-invariant by \cite[Theorem 2.5]{Malle08}.
\end{proof}

We are now ready to prove Theorem \ref{thm:simplegroups}.

\begin{proof}[Proof of Theorem \ref{thm:simplegroups}]
First, assume $P$ is cyclic. Then by Dade's cyclic-defect theory \cite[Theorem 1]{Dade66},
we have $k(B_0)=f+\frac{|P|-1}{f}$, where
$f:=[\norm{S}{P}:\cent{S}{P}]$. This forces $|P|=9$ when $k(B_0)=6$.

Hence we assume that $P$ is non-cyclic, and for contradiction, we
assume that $k(B_0)=6$ and $|P|\neq 9$. By \cite[Theorem
1.1]{Hung-SchaefferFry}, we know that $k(B_0)\geq 2\sqrt{p-1}$, and
hence we may further assume that $p\in\{3,5,7\}$, so that $|P|\geq
25$.

As before, the cases of sporadic groups, the Tits group,
$\tw{2}\type{G}_2(3)'$, and groups of Lie type with exceptional
Schur multiplier can be seen using GAP.

Next suppose $S=\Alt_n$ is a simple alternating group and let
$n=pm+r$ with $m, r$ integers satisfying  $m\geq 1$ and $0\leq r<p$.
Then \[k(B_0)\geq
\frac{1}{2}k(B_0(\Sym_n))=\frac{1}{2}k(B_0(\Sym_{pm}))\] where the
equality comes from  \cite[Theorem (1.10)]{MO83}.  Note that our
assumption $|P|\geq 25$ forces $m\geq 2$ if $p\geq 5$ and $m\geq 3$
if $p=3$.  In these cases, we can explicitly find at least 13
partitions of $pm$ with trivial $p$-core, so that
$k(B_0(\Sym_{pm}))\geq 13$ and $k(B_0)>6$.

Finally, let $S$ be a simple group of Lie type defined over
$\mathbb{F}_q$ with $q$ a power of some prime.  If $p\mid q$, then
the condition $P$ is abelian again leaves only $S=\PSL_2(q)$, and
the condition  $|P|\geq 25$ forces $q\geq 25$. Now, by \cite[Theorem
3.3]{Cabanes18}, we have $k(B_0)=k(S)-1>6$, using the well-known
character table for $S$.

Hence we assume $S$ is a simple group of Lie type defined over
$\mathbb{F}_q$ with $p\nmid q$.  We remark that our proof in this
remaining case follows  the work in \cite{RSV21} closely, although
we need to exhibit one more character here than we needed there.

Assume first that $S$ is of exceptional type. Note that
$\tw{2}\type{B}_2(q)$ and $\tw{2}\type{G}_2(q)$ have cyclic Sylow
$p$-subgroups,  $B_0(\tw{2}\type{F}_4(q))$ has more than 6
characters, using \cite{malle91} and $B_0(\tw{3}\type{D}_4(q))$ has
more than 6 characters using \cite{DeM87}, so we further assume that
$S$ is not of Suzuki, Ree, or triality type. If the order $d_p(q)$
of $q$ modulo $p$ is not a so-called regular number, then (with our
previous assumptions) we see from \cite[Table 2]{Broue-Malle-Michel}
that $B_0(S)$ contains more than 6 characters. So, assume that
$d_p(q)$ is regular, so that every $p'$-degree unipotent character
lies in $B_0(S)$ (see e.g. \cite[Lemma 3.6]{RSV21}). For
$S=\type{G}_2(q)$, we see using \cite{shamash} that $B_0(S)$
contains more than 6 characters.  In the remaining exceptional
groups $\type{F}_4(q)$, $\type{E}_6^\pm(q)$, $\type{E}_7(q)$, and
$\type{E}_8(q)$, we see using \cite[Section 13.9]{Carter85} that
there are more than 6 $p'$-degree unipotent characters when $P$ is
noncyclic and abelian and $d_p(q)$ is regular, so we are done with
the exceptional groups.

We are left with the case that $S$ is a finite classical group. If
$S$ is $\PSL_n^\epsilon(q)$ with $\epsilon\in\{\pm\}$ and $n\geq 2$,
let $e$ be $d_p(\epsilon q)$.  If $S$ is $\POmega_{2n}^\epsilon(q)$
with $n\geq 4$,  $\POmega_{2n+1}(q)$ with $n\geq 3$, or
$\PSp_{2n}(q)$ with $n\geq 2$, let $e$ be $d_p(q^2)$. Arguing as in
\cite[Sections 6 and 7]{Hung-SchaefferFry}, we have the number of
unipotent characters in $B_0(S)$ is at least $k(e,m)$, $k(2e,m)/2$,
$k(2e,m)$, and $k(2e,m)$, in the cases if $S=\PSL_{n}^\epsilon(q)$,
$\POmega_{2n}^\epsilon(q)$,  $\POmega_{2n+1}(q)$, or $\PSp_{2n}(q)$,
respectively. Here $m$ is such that $n=em+r$ with $0\leq r<e$ and
$k(s,t)$ can be computed as in \cite[Lemma 1]{Olsson84}. Further,
our assumption that $P$ is abelian but not cyclic forces $p\geq
m>1$. With this and our assumption that $P$ is not cyclic, we have
the number of unipotent characters in $B_0(S)$ is at least 7 unless
possibly if $(e,m)\in\{(1,3), (1,4), (2,2)\}$ if
$S=\PSL_n^\epsilon(q)$ or $(e,m)\in\{(1,2), (1,3)\}$ otherwise. In
the latter cases, since $e=1$, we have $n=m$, so that
$S=\PSp_{2n}(q)$ or $\POmega_{2n+1}(q)$, and the case $(e,m)=(1,3)$
gives 10 unipotent characters in $B_0(S)$. Then we may assume
$S=\PSL_n^\epsilon(q)$ with $(e,m)\in\{(1,3), (1,4), (2,2)\}$ or
$S=\PSp_4(q)$ with $(e,m)=(1,2)$.  Let $G=\SL_n^\epsilon(q)$ and
$\wt{G}=\GL_n^\epsilon(q)$ in the first cases and
$G=\wt{G}=\Sp_{4}(q)$ in the latter.

First let $(e,m)\neq (1,3)$, so that there are at least 5 unipotent
characters in $B_0(S)$. We may consider two semisimple characters
$s_1, s_2$ of $\wt{G}$ corresponding to semisimple elements of
$\wt{G}^\ast\in\{\GL_n^\epsilon(q), \SO_5(q)\}$ with nontrivial
eigenvalues $\{a, a^{-1}, b, b^{-1}\}$, where $a,b\in
C_{(q^2-1)_p}\leq \FF_{q^2}^\times$ satisfy $a\neq b$ in the case of
$s_1$ and $a=b$ in the case of $s_2$. These necessarily give
distinct characters on restriction to $G$ since $s_1z$ is not
conjugate to $s_2$ for any $1\neq z\in\zent{\wt{G}^\ast}$ and are
trivial on $\zent{G}$, as they lie in $[\wt{G}^\ast, \wt{G}^\ast]$,
which yields two more members of $\irr{B_0(S)}$.

Finally, let $(e,m)=(1,3)$, so we have $e=1$, $p>3$, and
$S=\PSL_3^\epsilon(q)$. Here $B_0(S)$ contains all 3 unipotent
characters and therefore $B_0(G)$ contains all characters lying in
Lusztig series indexed by semisimple elements $s\in G^\ast$ with
$|s|$ a power of $p$ (see \cite[Theorem 9.12]{CE04}). We may obtain
three semisimple elements $s_1, s_2, s_3$ with eigenvalues $\{a, a,
a^{-2}\}$, $\{a,b,(ab)^{-1}\}$, and $\{a,a^{-1},1\}$ with $a,b\in C_{q-\epsilon}$;
$|a|=p=|b|$; and $a\neq b$, whose Lusztig series give distinct
restrictions to $G$ and are trivial on $\zent{G}$ arguing as before.
Further, the series $\mathcal{E}(\wt{G}, s_1)$ contains two characters, corresponding to the two unipotent characters of $\cent{\wt{G}}{s_1}\cong \GL_2^\epsilon(q)\times \GL_1^\epsilon(q)$.
This yields at least 4 additional characters in $B_0(S)$, and we are
done.
%
%dumb notes:
%$k(2,2)=5$,
%$k(3,2)=9$,
%$k(4,2)=14$,
%$k(2,3)=10$,
%$k(2,4)=20$,
%$k(e,w)\geq k(1,w)=\pi(w)$ and $\pi(w)=1,2,3,5,7,11,15$ for $w=1,2,3,4,5,6,7$ (larger than 6 for $w\geq 5$)
\end{proof}

%%%%%%%%%%%%%%%%%%%%%%%%%%%%%%%%%%%%%%%%%%%%%%%%%%%%%%%%%%%%%%

\section{Principal blocks with six irreducible
characters}\label{sec:3}

The purpose of this section is to prove Theorems \ref{thm:kB0=6} and
\ref{thm:kB0=6-main}.

\subsection{Preliminaries} When $p$ is a prime, we
write $k_p$ for the $p$-part of an integer $k$. Also, $l(B)$ denotes
the number of irreducible Brauer characters in a block $B$.

\begin{lemma}\label{lem:brauercharacters}
Let $G$ be a finite group and $p$ a prime. Let $B_0$ denote the
principal $p$-block of $G$.
\begin{itemize}
\item[(i)]  $k(B_0)\geq|\zent G|_p l(B_0)$.
\item[(ii)]  $ l(B_0)=1$ if, and only if, $G$ has a normal $p$-complement.
\end{itemize}
\end{lemma}
\begin{proof}
Part (i) is a direct consequence of \cite[Theorem 5.12]{nbook}. Part
(ii) is \cite[Corollary 6.13]{nbook}.
\end{proof}

\begin{lemma}\label{onlycovering}
Let $M \nor G$ and $P \in \syl p G$. If $P \cent G P \sbs M$, then
$B_0(G)$ is the only block of $G$ covering $B_0(M)$. In particular,
$k(G/M)<k_0(B_0(G))$ as long as $P>1$.
\end{lemma}

\begin{proof}
See \cite[Lemma 2.3]{HSV21}.
\end{proof}

We will also make use of Alperin-Dade's theory of isomorphic
principal blocks.

\begin{theorem}[Alperin-Dade correspondence]\label{isomblocks}
Suppose that $N$ is a normal subgroup of $G$, with $G/N$ a
$p'$-group. Let $P \in \syl p G$ and assume that $G=N\cent GP$. Then
restriction of characters defines a natural bijection between the
irreducible characters of the principal $p$-blocks of $G$ and $N$.
In particular, $k(B_0(G))=k(B_0(N))$.
\end{theorem}

\begin{proof}
The case where $G/N$ is solvable was proved in \cite{Alp76} and the
general case in \cite{Dad77}.
\end{proof}

Recall that by Gallagher's theorem, if $\theta\in \irr N$, where $N\nor G$, extends to $G$, then  $|\irr{G|\theta}|=k(G/N)$. The following is well-known
and guarantees the existence of an extension.

\begin{lemma}\label{extension}
Let $N \nor G$ and $\theta \in \irr N$ be $G$-invariant. If the
Sylow subgroups of $G/N$ are cyclic ({for all primes dividing $|G/N|$}), then $\theta$ extends to $G$.
\end{lemma}

\begin{proof}
The statement follows by \cite[Theorem 11.7 and Corollary
11.21]{Isaacs}.
\end{proof}

We also need a version of \cite[Lemma 2.2]{HSV21}. We write $\irrp p
{G}$ to denote the subset of $\irr G$ consisting of those characters
with degree coprime to $p$.

\begin{lemma}\label{Gallagher}
Let $N \nor G$ and let $B_0=B_0(G)$ be the principal $p$-block of
$G$. If $\theta \in \irr {B_0(N)}$ extends to $\chi \in \irr{B_0}$
with $(\chi(1), p)=1$, then
$$|\irrp p {B_0 | \theta}|\geq k_0(B_0(G/N)), $$
 where $\irrp p {B_0 | \theta}=\irrp p G \cap \irr{B_0 | \theta}$.
\end{lemma}

\begin{proof}
By Gallagher's theorem, $\irrp p {B_0 | \theta}=\{ \beta \chi \in
\irr{B_0} \ | \ \beta \in \irrp p {G/N} \}$. By \cite[Lemma
2.3]{GRSS20} we have $\beta \chi \in \irr{B_0}$ whenever $\beta \in
\irr{B_0(G/N)}$, and the lemma follows.
\end{proof}

We note that the example of $G=\Sym_3$ and $N=\Alt_3=\langle (1 \, 2 \,3
)\rangle$ for $p=3$ shows that the inequality in
Lemma~\ref{Gallagher} is not always an equality.

%\begin{lemma}\label{mlem:invariant}
%Let $p$ be a prime and $S$ a nonabelian simple group. Suppose that
%$(S,p)\neq ({\rm PSL_2(3^n)},3)$. Then there exists a non-principal
%$\aut(S)$-invariant irreducible character in the principal $p$-block
%of $S$.
%\end{lemma}
%
%\begin{proof} This was proved in \cite[Proposition
%2.1]{GRSS20} for $p \geq 5$ and in \cite[Proposition 3.7]{Martinez}
%for $p=3$ and $S\neq {\rm PSL}_2(3^n)$ whenever $p\neq 3$ or $S \neq
%{\rm PSL}_2(3^n)$ with $n>2$.
%\end{proof}

%%%%%%%%%%%%%%%%%%%%%%%%%%%%%%%%%%%%%%%%%%%%%%%%%%%%%%%%%%%%%%%%%%

\subsection{Proof of Theorem \ref{thm:kB0=6}}
As we will see later, known results on height zero characters in
principal blocks essentially reduce Theorem \ref{thm:kB0=6} to the
following.

\begin{theorem}\label{thm:casek0(B0)=k(B0)=6}
Let $G$ a finite group of order divisible by $p \in \{ 3, 5, 7\}$.
Suppose that $P\in\Syl_p(G)$ is abelian and the principal $p$-block
$B_0$ of $G$ has precisely $6$ irreducible characters. Then $p=3$
and $|P|=9$.
\end{theorem}

\begin{proof}
Assume that the statement is false and let $G$ be a counterexample
of minimal order. In particular, $k(B_0)=6$ and $|P|\neq 9$. If
$|P|<9$ then $P$ would be ${\sf C}_3$, ${\sf C}_5$, or ${\sf C}_7$,
and Dade's cyclic-defect theory \cite[Theorem 1]{Dade66} quickly
shows that $k(B_0)\neq6$. We therefore indeed have
\[
|P|>9.
\]
Furthermore, by Lemma \ref{lem:blockabove}(ii), we have ${\bf
O}_{p'}(G)=1$.

We now show that $G$ is not $p$-solvable. Assume that $G$ is
$p$-solvable. Then, by Fong's theorem (see \cite[Theorem
10.20]{nbook}), $G$ has a unique $p$-block - the principal one, and
so $k(G)=k(B_0)=6$. An inspection of the list of finite groups with
6 conjugacy classes (\cite[Table 1]{VV85}) reveals no
counterexamples. Therefore $G$ is in fact not $p$-solvable.

Let $N\nor G$ be a minimal normal subgroup of $G$. As ${\bf
O}_{p'}(G)=1$, we have that $p$ divides the order of $N$, which
therefore is either an elementary abelian $p$-group or a semisimple
group.

\medskip

(A) We claim that $p$ does not divide $|G:N|$.

Assume otherwise. Let $\bar B_0:=B_0(G/N)$ be the principal block of
$G/N$. On one hand we have $k(\bar B_0)\leq 5$ by Lemma \ref{lem:blockabove} (i) and (iii);
on the other hand, $k(\bar B_0)\geq \lceil
2\sqrt{p-1}\rceil\geq 3$ by \cite[Theorem 1.1]{Hung-SchaefferFry}.
In summary,
\[k(\bar B_0)\in \{3,4,5\}.\]
Also notice that, by Lemma \ref{lem:blockabove}, we have that $
\irr{B_0}\cap \irr{G/N}$ is a union of blocks of $G/N$ including
$\bar B_0$. Hence, we can write
\[ \irr {B_0}=(\irr{B_0} \cap \irr{G/N})\cup \{ \chi_1, \ldots, \chi_l \} \, , \]
with $l\in \{ 1, 2, 3\}$, where the union is disjoint.

\smallskip

(A1) Assume that $l=3$. Then $\irr{B_0} \cap \irr{G/N}=\irr {\bar
B_0}$ and $k(\bar B_0)=3$. It follows from \cite[Theorem 3.1]{KS21}
that
\[p=3 \text{ and } |PN:N|=3.\] Moreover, the action of $G$ on $\irr
{B_0(N)}\setminus \{ {\bf 1}_N\}$ defines at most 3 orbits.

Suppose first that $N$ is elementary abelian. Then \[P\cent G
P=\cent G P \sbs \cent G N=:M\nor G.\] By Lemma~\ref{onlycovering},
$B_0$ is the only block of $G$ covering $B_0(M)$ and $1\leq
k(G/M)\leq 3$, as $\irr{G/M}\sbs \irr{B_0}\cap \irr{G/N}=\irr{\bar
B_0}$. Since $G/M$ is a $3'$-group, we necessarily have $|G/M|\leq
2$. If $G=M$, then $N \sbs \zent G$ and by Lemma
\ref{lem:brauercharacters}, $6=k(B_0)\geq |N|l(B_0)$, implying that
$|N|$ must be 3, so that $|P|=3|N|=9$, a contradiction. We therefore
must have $|G/M|=2$. As the action of $G/M$ on the nontrivial
elements of $ N$ defines at most 3 orbits, we have $|N|\leq 7$, and
thus $|N|=3$, which implies that $|P|=9$, a contradiction again.

Now suppose that $N\cong S^t$ for some non-abelian simple group $S$ of order
divisible by $p=3$
and $t\in\ZZ^+$. By \cite[Theorem~2.2]{Martinez}, $\aut (S)$
produces at least 2 orbits on $\irr{B_0(S)}\setminus \{ {\bf 1}_S \}
$. Since the action of $G$ on $\irr {B_0(N)}\setminus \{ {\bf
1}_N\}$ defines at most 3 orbits, we deduce that  $t=1$ and $N=S$ is
simple (otherwise, assuming that $\alpha, \beta\in
\irr{B_0(S)}\setminus \{ {\bf 1}_S \}$ lie in different
$\aut(S)$-orbits, then $\alpha^t$, $\alpha\times {\bf 1}_S^{t-1}$,
$\beta^t$ and $\beta \times {\bf 1}_S^{t-1}$ would lie in different
$G$-orbits). We then have $|P:(P\cap S)|=|PS:S|=3$. The fact that
$G$ is a counterexample then implies that $Q:=P\cap S \in \syl 3 S$
has order at least 9.

We now use Theorem \ref{thm:simplegroups3}, with the almost simple
group $G/\bC_G(S)$ in place of $A$, to arrive at one of the
following cases.

\begin{enumerate}
\item[(a)] The action of $G$ on $\irr{B_0(S)}\setminus \{ {\bf 1}_S \}$ defines at least 4 orbits, a contradiction.

\item[(b)] $S \in \{ {\rm PSL}_2(q), {\rm PSL}_3(q), {\rm PSU}_3(q) \}$ with  $(3, q)=1$,
and $3$ divides $|G:S\cent G S|$. Since $|G:S|_3=3$, $\cent G S$ has
order not divisible by 3. Then $\cent G S=1$, because $\oh{3'}G=1$ and $G=A$ is almost
simple. Theorem \ref{thm:simplegroups3}(b) yields $k(B_0)>6$, a
contradiction.

\item[(c)] $S=\Alt_6$ and $A:=G/\cent G S$ has a subgroup isomorphic to ${\rm
M}_{10}$. It follows that $A$ is either $M_{10}$ or the full
automorphism group $\aut(S)$. In fact, as $k(B_0(A))\leq k(B_0)=6$
and $k(B_0(\aut(S)))>6$, we have $A\cong {\rm M}_{10}$, in which
case $k(B_0(A))=k(B_0)=6$. We conclude that 3 does not divide the
order of $\cent G S$ (otherwise $B_0(\cent G S)$ contains a
nontrivial character, say $\theta$, and by
Lemma~\ref{lem:blockabove}(iii), $B_0(G)$ would have a character
lying above $\theta$, implying that $k(B_0(A))<k(B_0)$). In
particular, $|G|_3=|\Alt_6|_3=9$ and $G$ is not a counterexample.
\end{enumerate}

\smallskip

(A2) Assume that $l=2$. If $\irr{\bar B_0}\subsetneq \irr{B_0}\cap
\irr{G/N}$ then $\irr{B_0}$ would contain a block of $p$-defect zero
of $G/N$, but that is impossible as every character in $\irr{B_0}$
has degree coprime to $p$ by \cite{KM13}. Thus $\irr{\bar B_0}= \irr{B_0}\cap
\irr{G/N}$. In particular $k(\bar B_0)=4$ and,
by \cite[Theorem 1.1]{KS21}, \[p=5 \text{ and } |PN:N|=5.\] In this
case the action of $G$ on $\irr {B_0(N)}\setminus \{ {\bf 1}_N\}$
defines at most 2 orbits.

Suppose that $N$ is elementary abelian. Then, as above,  $P\sbs \cent G
N:=M\nor G$. By Lemma~\ref{onlycovering}, $B_0$ is the only block covering $B_0(M)$.
Consequently, Lemma \ref{lem:blockabove}(iv) implies that $\irr{G/M}\sbs \irr{\bar B_0}$ and so
$1\leq k(G/M)\leq4$. Moreover, using Lemma \ref{lem:blockabove} (iii) and (iv),
considering a non-principal character $\theta \in \irr{B_0(M/N)}\sbs
\irr{B_0(M)}$, we in fact have that $k(G/M)<4$. If $G=M$ then $N \sbs
\zent G$ and, by Lemma \ref{lem:brauercharacters}(i), $6=k(B_0)\geq
5l(B_0)$. It follows that $l(B_0)=1$ and, by Lemma
\ref{lem:brauercharacters}(ii), $G$ would be $p$-solvable, a
contradiction. Hence the the action of the nontrivial group $G/M$ on
$N$ is faithful and defines at most 2 orbits on the nontrivial
elements of $N$. The only possibility is that $|N|=5$ and $|G/M|=2$.
Since $\irr{B_0(M)}$ consists of the irreducible constituents of
$\psi_M $ for $\psi \in \irr{B_0}$, the fact $|G/M|=2$ forces
$k(B_0(M))\in \{ 3, 6, 9\}$. As $p=5$, the main result of
\cite{Hung-SchaefferFry} implies that $k(B_0(M))\geq 2\sqrt{p-1}=4$.
Furthermore, $k(B_0(M))$ cannot be 6 neither by the minimality of
$G$. So $k(B_0(M))=9$. Lemma \ref{lem:brauercharacters} then implies
that $l(B_0(M))=1$ and $M$ is $p$-solvable. Thus $G$ is $p$-solvable
as well, a contradiction.

Suppose that $N \cong S^t$ for some non-abelian simple group $S$ of
order divisible by $p=5$ and $t \in \ZZ^+$. If $t>1$, then, as before, \cite[Theorem
2.2]{Martinez} shows that $G$ defines more than 2 orbits when acting
on $\irr {B_0(N)}\setminus \{ {\bf 1}_N\}$, violating the fact that
$l=2$. Hence $N=S \nor G$.

Assume first that $S\neq P\Omega_8^+(q)$. By \cite[Proposition
2.1(i)]{GRSS20}, there are some ${\bf 1}_S\neq\theta \in
\irr{B_0(S)}$ that extends to $\hat \theta \in \irr{B_0(G)}$ and
${\bf 1}_S\neq\varphi \in \irr{B_0(S)}$ with $\varphi(1) \nmid
\theta(1)$. Furthermore, both $\theta$ and $\varphi$ have degree
coprime to $p$. Recall that $k(B_0(G/S))=k(B_0(G/N))\geq 3$. We
therefore can write
\[\irr{B_0}=\{ {\bf 1}_G, \alpha, \beta, \gamma, \chi_1, \chi_2 \}\]
so that $\alpha, \beta, \gamma$ contain $S$ in their kernels,
$\chi_1=\hat \theta$, and $\chi_2$ lies over $\varphi$. In
particular, $\chi_1$ is the only character in $\irr{B_0}$ that is
above $\theta$. This together with the facts that $p\nmid \theta(1)$
and $p\mid |G:S|$ would contradict Lemma~\ref{Gallagher}.

We are left with the case $S= {\rm P}\Omega_8^+(q)$. Then by
\cite[Proposition 2.1.(ii)]{GRSS20}, the set $\irr{B_0(S)}\setminus
\{ {\bf 1}_S \}$ contains two $\aut (S)$-invariant members. Since
$k(B_0(S))\geq 2\sqrt{p-1}=4$ by \cite[Theorem
1.2.(i)]{Hung-SchaefferFry}, it follows that the action of $G$ on
$\irr{B_0(S)}\setminus \{ {\bf 1}_S\}$ defines at least 3 different
orbits, which is again a contradiction.

\smallskip

(A3) Assume that $l=1$. If $\irr{\bar B_0}\subsetneq \irr{B_0}\cap
\irr{G/N}$ then arguing as in the second sentence of case (A2) we
see that $\irr {B_0}$ contains a block $\bar B_1$ of $G/N$ with
$k(\bar B_1)=2$. But $k(\bar B_1)=2$ would force $p=2$ by
\cite[Theorem A]{Brandt}, contradicting our hypothesis on $p$. Hence
$k(\bar B_0)=5$. Using the main result of \cite{RSV21}, we then have
\[p\in \{ 5, 7 \} \text{ and } |PN:N|=p.\]

In this case all the characters in $\irr {B_0(N)}\setminus \{ {\bf
1}_N\}$ must be $G$-conjugate (and lie under $\chi_1$). In
particular, the set of character degrees of $\irr{B_0(N)}$ has size
2, and it follows from \cite[Theorem A]{Martinez} that $N$ is
$p$-solvable. We therefore conclude that $N$ is an elementary
abelian $p$-group. Also, $N\subseteq P$ and $|P/N|=p$.

As before, $B_0$ is the only block covering $B_0(M)$, where
$M:=\cent G N\nor G$. Moreover, $k(G/M)<5$. If $G=M$, then $N \sbs \zent
G$ and Lemma \ref{lem:brauercharacters}(i) implies that
$6=k(B_0)\geq |N|l(B_0)$, forcing $l(B_0)=1$ as $p\geq 5$, and so
$G$ would be $p$-solvable by Lemma~\ref{lem:brauercharacters}(ii), a
contradiction. Therefore the $p'$-group $G/M\in \{ {\sf C}_2, {\sf
C}_3, \Sym_3, {\sf C}_4, {\sf C}_2\times {\sf C}_2, {\sf D}_{10},
\Alt_4\}$
 acts faithfully on $N$ and transitively on $N\setminus \{1\}$.
 (See \cite{VV85} for the list of finite groups of relatively small class number.)
There are only two possible scenarios: $|G/M|=4$ and $|N|=5$ or
$G/M\cong {\Sym}_3$ and $|N|=7$. The latter case in fact cannot
happen as $\aut({\sf C}_7)\cong {\sf C}_6$. In the former case,
since $\irr{G/M}\sbs \irr{\bar B_0}$ and $k(\bar B_0)=5$, the
principal block $\bar B_0$ of $G/N$ would have only 2 different
character degrees. It again follows from the main result of
\cite{Martinez} that $G/N$, and therefore $G$, would be
$p$-solvable, a contradiction.

\bigskip

(B) We have shown that $p$ does not divide $|G:N|$. Therefore we
have $P\leq N$. In particular, as $G$ is not $p$-solvable, $N$ is
isomorphic to a direct product $S^t$ of $t$ copies of a non-abelian
simple group $S$ of order divisible by $p$.

Let $M:=N \cent G P$. Then $M \nor G$ by the Frattini argument, and
\[k(B_0(M))=k(B_0(S))^t\] by Theorem \ref{isomblocks}. Moreover,
$G/M$ transitively permutes the simple factors of $N$. By
Lemma~\ref{onlycovering}, $B_0$ is the only block of $G$ covering
$B_0(M)$. In particular, $\irr{G | \phi}\sbs \irr{B_0}$ for every
$\phi \in \irr{B_0(M)}$. Moreover, as $k(B_0)=6$, we have $1\leq
k(G/M)<6$.

If $G=M$ then $t=1$ and $k(B_0(S))=k(B_0(G))=6$, and it follows from
Theorem \ref{thm:simplegroups} that $|P|=9$, contradicting the
choice of $G$ as a counterexample. If $k(G/M)=5$ then all members of
$\irr{B_0(M)}\setminus \{ {\bf 1}_M \}$ lie under the same character
in $\irr{B_0}$. In particular, \cite[Theorem A]{Martinez} implies
that $M$ is $p$-solvable, and so would be $G$, which is not the
case. We have shown that
\[k(G/M)\in\{2,3,4\}.\]

Recall that $|P|>9$. Therefore, by previous results on possible
Sylow structure of finite groups with up to $5$ characters in the
principal block (see Theorems 4.1--4.5 of \cite{HSV21}), we have
$k(B_0(M))\geq 6$. Note that $P\in\Syl_p(M)$, and so the minimality
of $G$ as a counterexample further implies that
\[k(B_0(M))\geq 7.\]

Let us first handle the special case $(S,p)= ({\rm PSL}_2(3^m),3)$
for some $m\geq 3$. If $S={\rm PSL}_2(27)$ then $\aut(S)$ has $5$
orbits on $\irr {B_0(S)}\backslash \{ \textbf{1}_S\}$ (\cite{Atl1}), so
$k(B_0)\geq 5+k(G/M)\geq 7$ and we would be done. For $S={\rm
PSL}_2(q)$ with $q\geq 81$ we have $k(B_0(S)) =(q+3)/2 \geq 42$ (the
ordinary characters of $B_0(S)$ are $\Irr(S)\backslash \{ \St_S\}$ and
$k({\rm PSL}_2(q))=(q+5)/2$ for odd $q$). We therefore have at least
$41$ nontrivial members in $\Irr(B_0(M))$. With $k(G/M) \leq4$, we
have $|G/M|\leq12$. So $G$ produces at least $\lceil 41/12\rceil=4$
orbits on $\Irr(B_0(M)) \backslash\{ \textbf{1}_M\}$. As $k(B_0)=6$, it
follows that $k(G/M)\leq 2$, which would imply that $G$ now defines
more than $41/2$ orbits on $\Irr(B_0(M)) \backslash\{ \textbf{1}_M\}$, a
contradiction.

We may assume from now on that $(S,p)\neq ({\rm PSL}_2(3^m),3)$ for
every $m\geq 3$. Then $\irr{B_0(S)}\setminus \{ {\bf 1}_S \}$
contains an $\aut(S)$-invariant member $\phi$ by \cite[Proposition
2.1]{GRSS20} for $p \geq 5$ and by \cite[Proposition 3.7]{Martinez}
for $p=3$ and $S\neq {\rm PSL}_2(3^n)$ for $n\geq 2$. (Note that the irreducible
character of degree 10 of $S={\rm Alt}_6\cong {\rm PSL}_2(9)$ is $\aut (S)$-invariant
and lies in the principal 3-block.) In particular, by Theorem
\ref{isomblocks}, there exists some \[{\bf 1}_M\neq \psi \in
\irr{B_0(M)}\] that is $G$-invariant. More precisely, under the
Alperin-Dade correspondence, $\psi$ corresponds to $\phi^t$ where
$\phi \in \irr{B_0(S)}\setminus \{ {\bf 1}_S \}$ is
$\aut(S)$-invariant.

\medskip

(B1) Assume that $k:=k(G/M)\in\{2,3\}$. Then all the Sylow subgroups of $G/M$ are cyclic and
$\psi$ extends to $G$ by Lemma \ref{extension}. Therefore
$\Irr(B_0)$ has $2k$ members lying over the characters
$\textbf{1}_M$ and $\psi$ in $\Irr(B_0(M))$. As $k(B_0(M))\geq 7$,
the remaining at least $5$ characters in $\Irr(B_0(M))$ produce at
least $\lceil 5/k\rceil$ additional irreducible characters in $B_0$.
We arrive at
\[
k(B_0)\geq 2k+\lceil 5/k\rceil,
\]
which is greater than 6 when $k\in\{2,3\}$, a contradiction.

\medskip

(B2) Finally we consider $k(G/M)=4$. Now $\Irr(B_0)$ contains four
members lying above $\textbf{1}_M$ and at least one member above
$\psi$. Furthermore, the set $\Irr(B_0(M))\backslash
\{\textbf{1}_M,\psi\}$, which again has cardinality at least 5,
would cover at least two $G$-orbits. As a result, we have
\[k(B_0)\geq 4+1+2=7,\] and this contradiction completes the proof.
\end{proof}

The following result on $2$-blocks with metacyclic defect groups
will be useful. The weaker case of maximal-class defect groups,
which is what we really need, is due to Brauer \cite{Brauer72} and
Olsson \cite{Ols75}.

\begin{lemma}\label{lem:Sambale}
Let $B$ be a $2$-block of a finite group with metacyclic defect
group $D$ of order at least $8$. Then either $k(B)\geq 7$ or one of
the following holds:
\begin{enumerate}[{\rm (i)}]
\item $B$ is nilpotent and $k(B)=k(D)$.

\item $D={\sf D}_8$ and $k(B)=5$.
\end{enumerate}
\end{lemma}

\begin{proof}
This follows from \cite[Theorem 8.1]{Sam14}.
\end{proof}

\begin{theorem}\label{thm:kB0=6repeated}
Let $G$ a finite group, $p$ a prime, and $P\in\Syl_p(G)$. Suppose
that the principal $p$-block of $G$ has precisely six irreducible
characters. Then $p=3$ and $|P|=9$.
\end{theorem}

\begin{proof}
Recall that $B_0$ denotes the principal $p$-block of $G$. First, as
$k(B_0)=6>1$, we know that $G$ has order divisible by $p$. As
before, let $k_0(B_0)$ denote the number of height zero characters
in $B_0$. Obviously $k_0(B_0)\leq 6$ and moreover $k_0(B_0)\geq 2$
by \cite[Problem 3.11]{nbook}. By \cite[Theorems 4.2 and 4.3]{HSV21}
we further have $k_0(B_0)\geq 4$. The case $k_0(B_0)=5$ cannot
happen since otherwise, by \cite[Theorem 1.2(C)]{HSV21}, $P$ is
cyclic, and thus $k(B_0)=k_0(B_0)=5$ by the main result of
\cite{KM13}, which is a contradiction. We now have
\[k_0(B_0)\in\{4,6\}.\]

Suppose that $k_0(B_0)=4$. Then it follows from \cite[Theorem
1.2(B)]{HSV21} that $p=2$ and $P$ has maximal class (so $P$ is
either dihedral, semidihedral, or generalized quaternion). Since
$k_0(B_0)<k(B_0)$, $P$ is nonabelian by again \cite{KM13}. In
particular, $P$ is a metacyclic group of order at least $8$. Using
Lemma \ref{lem:Sambale}, we deduce that $B_0$ is nilpotent and
$6=k(B_0)=k(P)$. However, according to \cite[Table 1]{VV85}, there
is no $p$-group with precisely six conjugacy classes.

We are left with the case $k_0(B_0)=k(B_0)=6$, and so $P$ is abelian
by \cite{MN21}. Using \cite[Corollary 1.3.(i)]{Landrock81}, we
deduce that $p$ must be odd. On the other hand, the main result of
\cite{Hung-SchaefferFry} implies that $p\leq k(B_0)^2/4+1=10$, and
we conclude that $p\in\{3,5,7\}$. The result now follows from
Theorem \ref{thm:casek0(B0)=k(B0)=6}.
\end{proof}

\subsection{Proof of Theorem \ref{thm:kB0=6-main}}

\begin{theorem}\label{thm:kB0=6-main-repeated}
Let $G$ a finite group, $p$ a prime, and $P\in\Syl_p(G)$. Let $B_0$
denote the principal $p$-block of $G$. Then $k(B_0)=6$ if and only
if precisely one of the following happens:
\begin{itemize}
\item[(i)] $P= {\sf C}_9$ and $|\norm G P :\cent G P |=2$.

\item[(ii)] $P={\sf C}_3 \times {\sf C}_3$ and either $\norm G P /\cent G P \in\{{\sf C}_4, {\sf
Q}_8\}$ or $\norm G P /\cent G P\cong {\sf C}_2$ acts fixed-point
freely on $P$.

\end{itemize}
\end{theorem}

\begin{proof}
In either implication of the statement, by Theorem \ref{thm:kB0=6},
we have $p=3$ and $|P|=9$. In particular, $P$ is abelian and
$k_0(B_0)=k(B_0)$ by \cite{KM13}.

If $P$ is cyclic then $\norm G P / \cent G P $ is a $3'$-subgroup of
$\aut(P)\cong {\sf C}_6$. By Dade's cyclic-defect theory
\cite[Theorem 1]{Dade66},  we have $k(B_0)=8/f+f$, where $f:=|\norm
G P / \cent G P|$. The only possibilities are that $k(B_0)=6$ with
$f=2$ or $k(B_0)=9$ with $f=1$.

If $P$ is elementary abelian, then Brou\'{e}'s abelian defect
conjecture holds for $B_0$ by \cite{Kos02}. In particular, we have
$k_0(B_0)=k_0(B_0(\bN_G(P)))$. Since every irreducible character of
$\bN_G(P)$ has $p'$-degree,
\[k_0(B_0(\bN_G(P)))=k(B_0(\bN_G(P)))=k(\bN_G(P)/\bO_{p'}(\bN_G(P))),\]
where the second equality follows from Lemma
\ref{lem:blockabove}(ii). In summary,
\[k(B_0)=k(\bN_G(P)/\bO_{p'}(\bN_G(P))).\] Note that, as $P$ is
abelian, we have $\bC_G(P)\cong \bO_{p'}(\bN_G(P))\times P$, so that
$\bN_G(P)/\bO_{p'}(\bN_G(P))$ is a semidirect product of the
$3'$-group
\[A:=\bN_G(P)/\bC_G(P)\] acting faithfully (and coprimely) on $P$.
Since  $\Aut(P)\cong \GL_2(3)$, there are eight possibilities for
the $3'$-group $A \leq \GL_2(3)$: \[1, {\sf C}_2, {\sf C}_4, {\sf
C}_2\times {\sf C}_2, {\sf C}_8, {\sf D}_8, {\sf Q}_8, \text{ and }
{\sf SD}_{16}.\] (Here ${\sf SD}_{16}$ is the semi-dihedral group of
order $16$.) While ${\sf C}_2$ can either invert every element of
$P$ or invert just elements in one factor ${\sf C}_3$ and fix
elements of the other, each of the other possibilities can only act
on $P$ in a unique way. Using \cite{GAP}, we easily check that the
corresponding semidirect products have 6 or 9 conjugacy classes.
Indeed, a straightforward inspection of \cite[Table 1]{VV85} reveals
that $k(PA)=6$ if and only if $A\in\{{\sf C}_4, {\sf Q}_8\}$ or
$A={\sf C}_2$ acts on $P$ by inversion. This completes the proof.
\end{proof}

We conclude the paper with a summary of what is now known on the
problem of classifying defect groups of blocks with a ``relatively small'' number
of irreducible ordinary characters.
In Table \ref{table:1}, as usual, we have used $B$ for an arbitrary
block and $B_0$ for a principal one. Also, $k(B)$ and $l(B)$
 denote the numbers of irreducible ordinary and Brauer
characters, respectively, in $B$. A defect group of $B$ is denoted by $D$ and, in
the case of principal blocks, by $P$. Moreover, a root of $B$ is
denoted by $b$ -- a block of $D\bC_G(D)$ with defect group $D$ that induces $B$,
see \cite[Theorem 9.7 and p. 198]{nbook}. The second column
presents all the possible defect-group structures (for a defect group
$D$) of a block with corresponding given values of $k(B)$ and $l(B)$
in the first column. The third column presents the ``if and only if"
conditions for those values to be achieved, with a note that a
missing item indicates no conditions being needed; as usual $T(b)$ denotes the stabilizer
of the block $b$ as in \cite[p. 193]{nbook}.

\begin{table}[ht]
\caption{Defect groups of small blocks\label{table:1}}
\begin{center}
\begin{tabular}{llll}
\hline

\begin{tabular}{l}$k(B),l(B)$ \end{tabular}& \begin{tabular}{l} Defect group \end{tabular}&
\begin{tabular}{l}Local condition \end{tabular}& Reference \\

\hline

\begin{tabular}{l}$k(B)=1$\end{tabular}& \begin{tabular}{l} $1$\end{tabular}&  & well-known  \\
\hline
\begin{tabular}{l}$k(B)=2$\end{tabular}& \begin{tabular}{l} ${\sf C}_2$ \end{tabular}  & & \cite{Brandt} \\
\hline
\begin{tabular}{l}$k(B_0)=3$\end{tabular}& \begin{tabular}{l} ${\sf C}_3$ \end{tabular} & & \cite{Belonogov,KS21} \\

\hdashline

\begin{tabular}{l} $k(B)=3$,\\ $l(B)=1$ \end{tabular}
& \begin{tabular}{l} ${\sf C}_3$ \end{tabular} & & \cite{Kul84} \\

\hdashline

\begin{tabular}{l} $k(B)=3$,\\ $D\nor G$ \end{tabular}
& \begin{tabular}{l} ${\sf C}_3$ \end{tabular} & & \cite{KNST14} \\

\hline

\begin{tabular}{l}$k(B_0)=4$\end{tabular}&  \begin{tabular}{l} ${\sf C}_2\times {\sf C}_2$\\
${\sf C}_4$\\ ${\sf C}_5$\end{tabular}
& \begin{tabular}{l} \\ \\ $|\bN_G(P)/\bC_G(P)|=2$ \end{tabular} & \cite{KS21} \\

\hdashline

\begin{tabular}{l} $k(B)=4$,\\ $l(B)=1$ \end{tabular}
& \begin{tabular}{l} ${\sf C}_2\times {\sf C}_2$\\ ${\sf C}_4$ \end{tabular} & & \cite{Kul84} \\

\hdashline

\begin{tabular}{l} $k(B)=4$,\\ $D\nor G$ \end{tabular}&  \begin{tabular}{l} ${\sf C}_2\times {\sf C}_2$\\
${\sf C}_4$\\ ${\sf C}_5$\end{tabular}
& \begin{tabular}{l} \\ \\ $|T(b)/\bC_G(D)|=2$ \end{tabular} & \cite{MRS22} \\

\hline

\begin{tabular}{l}$k(B_0)=5$\end{tabular}&  \begin{tabular}{l} ${\sf C}_5$\\ ${\sf C}_7$\\ ${\sf
D}_8$\\ ${\sf Q}_8$ \end{tabular}&

\begin{tabular}{l}
$|\bN_G(P)/\bC_G(P)|\in\{1,4\}$\\
$|\bN_G(P)/\bC_G(P)|\in\{2,3\}$\\ \\
$|\bN_G(P)/P\bC_G(P)|=1$
\end{tabular}
& \cite{RSV21} \\

\hdashline

\begin{tabular}{l} $k(B)=5$, \\ $l(B)=1$ \end{tabular}
& \begin{tabular}{l} ${\sf C}_5$\\ ${\sf D}_8$\\ ${\sf
Q}_8$\end{tabular} & \begin{tabular}{l} \\ $B$ is nilpotent.\\
$B$ is nilpotent.\end{tabular}
& \cite{CK92} \\

\hline

\begin{tabular}{l}$k(B_0)=6$\end{tabular}&

\begin{tabular}{l} ${\sf C}_9$ \\
{}\\
${\sf C}_3\times {\sf C}_3$\\ { }
\end{tabular} &

\begin{tabular}{l}
$|\bN_G(P)/\bC_G(P)|=2$\\

$\bN_G(P)/\bC_G(P)\in\{{\sf C}_4, {\sf Q}_8\}$ or\\
$\bN_G(P)/\bC_G(P)={\sf C}_2$\\
acts on $P$ by inversion.
\end{tabular}

& Theorem \ref{thm:kB0=6-main}\\

\hline

\begin{tabular}{l} $k(B_0)=7$,\\ $l(B_0)=6$ \end{tabular}
& \begin{tabular}{l} ${\sf C}_7$ \end{tabular}
& \begin{tabular}{l} $|\bN_G(P)/\bC_G(P)|=6$\end{tabular}& \cite{HST21} \\

\hline
\end{tabular}
\end{center}
\end{table}

\newpage

%%%%%%%%%%%%%%%%%%%%%%%%%%%%%%%%%%%%%%%%%%%%%%%%%%%%%%%%%%%%%%%%%%%%%%%%%%%%%%%

\end{document}